\documentclass[12pt,a4paper,reqno]{amsart}
\pdfoutput=1
\usepackage{amsfonts}
\usepackage{amsthm}
\usepackage{amsmath}
\usepackage{amssymb}
\usepackage{mathabx}
\usepackage{bm}
\usepackage{amscd}
\usepackage[latin2]{inputenc}
\usepackage{t1enc}
\usepackage[mathscr]{eucal}
\usepackage{indentfirst}
\usepackage{graphicx}
\usepackage{graphics}
\usepackage{pict2e}
\usepackage{epic}
\numberwithin{equation}{section}
\usepackage[margin=2.9cm]{geometry}
\usepackage{epstopdf} 
\usepackage[colorlinks,linkcolor=blue]{hyperref}
\usepackage[capitalise]{cleveref}

\crefformat{equation}{(#2#1#3)}
\crefrangeformat{equation}{(#3#1#4) to~(#5#2#6)}

\usepackage[normalem]{ulem}

\allowdisplaybreaks

\theoremstyle{plain}
\newcommand{\dv}{\mathrm{div}}
\newcommand{\vu}{\mathbf{u}}

\newtheorem{thm}{Theorem}
\newtheorem*{thm*}{Theorem}
\newtheorem{prop}{Proposition}
\newtheorem{lem}{Lemma}
\newtheorem{cor}{Corollary}
\newtheorem{remark}{Remark}

\begin{document}



\title[Variational rotating solutions to non-isentropic E-P equations]{Variational rotating solutions to non-isentropic Euler-Poisson equations with prescribed total mass}

\author[Y.Yuan]{Yuan Yuan}
\address{South China Research Center for Applied Mathematics and Interdisciplinary Studies, South China Normal University, \\
	Guangzhou, Guangdong 510631, P. R. China. }
\email{yyuan2102@m.scnu.edu.cn}
\date{\today}



\begin{abstract}
This paper proves the existence of variational rotating solutions to the compressible non-isentropic Euler-Poisson equations with prescribed total mass.
This extends the  result of the isentropic case [Auchmuty and Beals, \textit{Variational solutions of some nonlinear free boundary problems}, Archive for Rational Mechanics and Analysis \textbf{43} (1971), no. 4, 255-271] to the non-isentropic case. 
Compared with the previous result of variational rotating solutions in non-isentropic case [Yilun Wu, \textit{On rotating star solutions to the non-isentropic Euler-Poisson equations}, Journal of
Differential Equations \textbf{259} (2015), no. 12, 7161-7198], to keep the constraint of a prescribed finite total mass, the author establishes a new variational structure of the non-isentropic Euler-Poisson equations.
\end{abstract}

\maketitle

\section{Introduction}
This paper concerns with the motions of the non-isentropic Newtonian rotating gaseous stars, which are described by the compressible full Euler-Poisson equations: 
\begin{equation}\label{FEP}
\left\{ 
\begin{aligned}
\rho_t+\dv_x(\rho \vu)&=0,\\
(\rho \vu)_t+\dv_x(\rho \vu\otimes\vu)+\nabla p+\rho \nabla\Phi&=0,\\
(\rho S)_t+\dv_x(\rho \vu S)&=0,\\
\Delta \Phi&=4\pi  \rho,
\end{aligned}
\right.
\end{equation} 
where $\rho$, $\vu$, $p$, $\Phi$ and $S$  represent the density, velocity, pressure, gravitational potential and entropy respectively;
$(t,x)=(t,x_1,x_2,x_3)\in \mathbb{R}^+\times\mathbb{R}^3$ are the time and space variables. 
The pressure $p$ is given by the equation of state,
\begin{equation}
\label{state}
p=p(\rho, S)=f(\rho)e^S.
\end{equation} 
Let $\Phi$ vanish as $x$ tends to infinity,
\begin{equation}
\label{B}
\Phi=-\int_{\mathbb{R}^3}|x-y|^{-1}\rho(y)\  dy=:-B\rho.
\end{equation}

There have been many studies on the compactly supported steady (i.e. time-independent) solutions  to \eqref{FEP},
which describe the equilibrium configurations of self-gravitating gaseous stars, 
such as white dwarf stars \cite{Chandrasekhar1957}. 
The isentropic spherically symmetric stationary solutions ($S=0,\rho(x)=\rho(|x|), \vu=0$) are often referred as the Lane-Emden solutions,
and their existence and linearly stability problem can be found in \cite{Chandrasekhar1957, Lin1997}. Recently the global nonlinear  well-posedness and stability theory have been studied extensively in \cite{Gu2016, Hadzic2016, Jang2008, Jang2014,  Jang2013, Liu2016a, Luo2014, Luo2016b, Luo2016a, Liu2019b, Liu2019c, Makino2015, Makino2017}.

For isentropic axisymmetric ($\rho(x)=\rho(\sqrt{x_1^2+x_2^2},x_3)$) rotating solutions, namely the ones rotating along the $x_3$-axis with prescribed angular momentum,
if in the cylindrical coordinates 
\begin{equation}\label{cylinder}
\eta(x)=\sqrt{x_1^2+x_2^2}~ \text{and} ~z(x)=x_3~,
\end{equation} and the angular momentum of the fluid per unit mass $\sqrt{L}$ is given, then $\vu=(-\frac{x_2\sqrt{L}}{\eta^2},\frac{x_1\sqrt{L}}{\eta^2}),0)$ ($\vu (\mathbf{0})=\mathbf{0}$), and 
the compressible Euler-Poisson equations reduce to 
\begin{equation}
\label{sss}
\nabla p=\rho\{\nabla B\rho+L(m_{\rho}(\eta))\eta^{-3} \mathbf{i_{\eta}}  \},
\end{equation}
where $\mathbf{i_{\eta}}$ is the unit vector in the radial direction in the cylindrical coordinates, and 
\begin{equation}
\label{m}
m_{\rho}(s)=\int_{\eta(y)<s}\rho(y) \ dy.
\end{equation}
Auchmuty and Beals \cite{Auchmuty1971} established the existence of steady rotating solutions by the variational method.
They proved that when $p=f(\rho)$ and $L$ satisfy the assumption (A1) and (A3) (see Section \ref{cond}),
the minimizer of
\begin{equation*}\label{E1}
\widehat{E}_1(\rho)=\int_{\mathbb{R}^3}[A(\rho)(x)+\frac{1}{2}\rho(x)L(m_{\rho}(\eta(x)))\eta(x)^{-2}-\frac{1}{2}\rho(x)B\rho(x)] \ dx,
\end{equation*}
where 
\begin{equation}\label{A}
A(s)=s\int_{0}^{s}f(t)t^{-2}\ dt,
\end{equation}
exists in an admissible $W$ (in which $\rho$ is axisymmetric, of total mass $M$ and the three integrals in $E(\rho)$ are finite). And it is compactly supported, and satisfies \eqref{sss} where $\rho>0$.
Later, Li \cite{Li1991} studied the uniformly rotating gaseous stars with prescribed constant angular velocity.
For non-rotating fluids, the energy minimizers must be symmetric decreasing after translation, namely, the shapes must be balls(\cite{Lieb1987}). 
For rotating gaseous stars, the shapes of these variational solutions are studied in \cite{Caffarelli1980, Chanillo1994}. They studied the smoothness of the free boundary of $\{ \rho=0\}$ and proved that axisymmetric rotating solutions  to have a upper number of connected components of $\{ \rho>0\}$. 
The estimtes of the diameter of the above rotating star solutions are studied in \cite{Friedman1980, Chanillo1994}.
More related results about the existence of magnetic rotating stars, or the ones with an inner hard core, regularities and stability of these variational rotating solutions can be found in \cite{Auchmuty1991, Chanillo2012, Federbush2014, Friedman1981, Jang2017, Luo2008, Luo2009, Rein2003, Wu2016}.
Recently, there are also existence results of rotating gaseous stars by means of an implicit function theorem, see \cite{Jang2019a,Jang2019,Strauss2018,Strauss2017,Strauss2019} and references therein.

Now we consider the non-isentropic steady solutions. Stationary solutions, i.e. $\vu=0$, are obtained by \cite{Deng2002}. 
For rotating ones,  if $S$ is spherically symmetric (i.e. $S(x)=S(|x|)$) and some other conditions of $S$ are assumed,  Luo and Smoller \cite{Luo2004} obtained the existence theorem of spherically symmetric solutions with uniformly rotating regular velocity for $\frac{6}{5}<\gamma<2$ , and the non-existence theorem for $1<\gamma<\frac{6}{5}$ by studying the reduced second order ODE. 
See \cite{Deng2006} for more results in the regime.
However, these spherically symmetric solutions in \cite{Luo2004} might not be the minimizers of the physical energy. 
Recently Wu \cite{Wu2015} proved the existence of variational rotating solutions to the full Euler-Poisson equations, 
but had to replacing the original constraint $\int_{\mathbb{R}^3} \rho=M$ in \cite{Auchmuty1971} by $\int_{\mathbb{R}^3}f_1w=P$, where  $w=\frac{\gamma}{\gamma-1}e^{\frac{\gamma-1}{\gamma}S}\rho^{\gamma-1}$ and $f_1$ denotes a function of the angular velocity. 
However, in order to study steady states of a fixed total mass of a rotating gaseous star with the prescribed mass angular momentum, the constraint of a fixed total mass  has to be kept.
In this paper, we impose a different variational structure which keeps the above constraint, and prove the existence of non-isentropic rotating solutions.

\vspace{0.3cm}

\textbf{Motivations:}
To find physical variational solutions to \eqref{sss} and $\eqref{FEP}_3$,
the first main difficulty is to find the variational structure for equation \eqref{sss} in the non-isentropic case, that is, to find physical and feasible functionals and admissible sets. 
The physical energy should be the sum of the internal energy, the kinetic energy and the gravitational potential energy.
The last two are not influenced by the entropy, and thus we adopt the functionals of $\widehat{E}_1(\rho)$ in \cite{Auchmuty1971} and just redefine the internal one.
The specific internal energy $e$ satisfies the thermodynamics identity
\begin{equation*}
d e=T(\rho,S) dS+\frac{p(\rho,S)}{\rho^2} d\rho,
\end{equation*}	
where $T$ is the temperature.
Therefore,
\begin{equation*}
e(\rho,S)=\int_0^{\rho}\frac{p(t,S)}{t^2} dt=\frac{A(\rho)}{\rho} e^S.
\end{equation*} 
Considering the simplest case, define $S=S(x)$ just as \cite{Luo2004}. Therefore, the physical energy $\widehat{E}_2(\rho)$ should be  
\begin{equation}
\label{E2}
\begin{aligned}
\widehat{E}_2(\rho)
=&\int_{\mathbb{R}^3} A(\rho)(x)e^S(x)\ dx+\frac{1}{2}\int_{\mathbb{R}^3} \rho(x)L(m_{\rho}(\eta(x)))\eta^{-2}(x) \ dx\\
&\qquad  - \frac{1}{2}\int_{\mathbb{R}^3} \rho(x)B\rho(x)\ dx.
\end{aligned}
\end{equation} 
However, \eqref{sss} can not be derived by taking the gradient of the corresponding Euler-Lagrange equations of $\widehat{E}_2(\rho)$, which means that the solutions of \eqref{sss} are not the minimizers of the physical energy functional $\widehat{E}_2(\rho)$.
Therefore, defining $S=S(x)$ may not be appropriate.

To overcome this difficulty, we define $S$ as a function of the mass variable $n$,
where $n$ could be a function of $x$ which also depends on $\rho$, just like $m=m_{\rho}(\eta(x))$ (see \eqref{n}).
This assumption is reasonable because 
by $\eqref{FEP}_1$ and $\eqref{FEP}_3$, 
\begin{equation*}
S_t+v\cdot\nabla S=0,
\end{equation*}
which indicates that the entropy is invariant along the particle path. 
Once having found the variational structure, we could follow the method in \cite{Auchmuty1971}: 
first prove the existence of minimizers of the energy functional in the class $W_{R}^b$, in which the support of the density is contained in the ball of radius $R$; 
then obtain uniform bounds in  $L^{\infty}$ norm and radii of these minimizers.

Here appears another main difficulty:
there is one more term $\int_{r_b(y)>r_b(x)}A(\rho(y))\cdot e^{S(n_{b,\rho}(r_b(y)))}S'(n_{b,\rho}(r_b(y)))\ dy$ in the potential function \eqref{eprime}.
This term has to be proved to be sufficiently small for all large $|x|$.
When $S$ is continuously differentiable, 
it suffices to prove $\int_{r_b(y)>r_b(x)}A(\rho(y))\ dy$ sufficiently small  for all large $|x|$.
To overcome this difficulty, we will modify the arguments in \cite{Rein2001, Luo2008}.

\vspace{0.3cm}
This paper is organized as follows: in section 2 we will present the results and discuss about the setting and assumptions in the remarks. 
In section 3 we prepare the inequalities and verify the variational structure.
The existence theorems of the rotating solutions are proved in section 4 and 5.

\section{Main Results}

Throughout this paper,
the physical energy $E(\rho,n)$ is defined as the sum of the internal energy, the kinetic energy and the gravitational potential energy:
\begin{equation}
\label{energy}
\begin{aligned}
E(\rho,n)=&\int_{\mathbb{R}^3} A(\rho(x))T(n(x))\ dx+\frac{1}{2}\int_{\mathbb{R}^3} \rho(x)L(m_{\rho}(\eta(x)))\eta^{-2}(x) \ dx\\
&\qquad - \frac{1}{2}\int_{\mathbb{R}^3} \rho(x)B\rho(x)\ dx,
\end{aligned}
\end{equation}
where $\eta$, $m_{\rho}$ and $B\rho$ are defined in \eqref{cylinder}, \eqref{m} and \eqref{B} respectively.
Let \begin{equation}
\label{rb}
r_b(x)=\sqrt{\eta^2(x)+\frac{z^2(x)}{b^2}}, \qquad \text{for some}\ b>0
\end{equation}
\begin{equation}
\label{n}
n_{b,\rho}(s)=\int_{r_b(y)\leq s}\rho(y)\ dy.
\end{equation}
and  define
\begin{equation*}
E_b(\rho)=E(\rho,n_{b,\rho}).
\end{equation*}
The expression of the potential function $E_b'$ is
\begin{equation}
\label{eprime}
\begin{aligned}
E_b'(\rho)(x)=& A'(\rho(x))T\Big(n_{b,\rho}(r_b(x))\Big)+\int_{r_b(y)>r_b(x)}A(\rho(y))T'\Big(n_{b,\rho}(r_b(y))\Big)\ dy \\
&\qquad \qquad \qquad +\int_{\eta(x)}^{\infty}L(m_{\rho}(s))s^{-3}\ ds-B\rho(x).
\end{aligned}
\end{equation}
The admissible set $W^b$ is defined as 
\begin{equation}
\label{w}
\begin{aligned}
W^b&=\{ \rho: \mathbb{R}^3\rightarrow \mathbb{R} | \; \rho(x)\geq0\ \text{a.e.},\ \int_{\mathbb{R}^3}\rho(x)\ dx=M,\ \rho(x)=\rho(r_b(x)) \\
&\text{and the three integrals in \eqref{energy} are all finite} \}. 
\end{aligned}
\end{equation}
Here $\rho(x)=\rho(r_b(x))$ implies that $\rho$ is a constant on every ellipsoid $x_1^2+x_2^2+x_3^2/b^2=C^2$. So we call it ellipsoidal symmetry in the following.
This is a stronger assumption than the axisymmetric one in \cite{Auchmuty1971}. We  will continue to discuss about it in \cref{rem-symmetry}.

\vspace{0.5cm}

Now assume the following conditions on $f$, $S$(entropy per mass) and $L$(the square of angular momentum per mass):
\begin{enumerate}
\item[(A1)] \label{cond} $f\in C^1(0,\infty)$, $f(s)\geq 0,~ f'(s)>0$ for $s>0$.
$\lim\limits_{s\rightarrow0}f(s)s^{-4/3}=0$, $\lim\limits_{s\rightarrow\infty}f(s)s^{-4/3}=\infty$.
\item[(A2)] 
$\lim\limits_{s\rightarrow\infty}f(s)s^{-\bar{\gamma}}=0$ for 
 some constant  $\bar{\gamma}>1$
\item[(A3)] 
$L(m)$ is non-negative and absolutely continuous for $0\leq m\leq M$, with $L(0)=0$.
\item[(A4)] $L(am)\geq a^{4/3}L(m)$ for $0\leq a\leq1$ and $0\leq m\leq M$.
\item[] $L'(m)\geq 0$ for $0\leq m\leq M$.
\item[(A5)] $S(n)\in C^1[0,M]$, and without loss of generality, we assume $S(0)=0$.
\item[(A6)] $|S'(n)|\leq \frac{2}{3M}$ for $0\leq n\leq M$. 
\item[(A7)] There exists $\delta_0>0$ such that $S'(n)\leq 0$ for $M-\delta_0\le n\leq M$.
\end{enumerate}

The condition (A1) and the definition of $A$ \eqref{A} imply that $A''(s)=f'(s)/s>0$ if $s>0$. Therefore $A$ is convex and $A'$ is strictly increasing.

For simplicity, in the following we denote $T:= e^S$, and thus $(A5)$ implies
that $T(n)$ is continuously differentiable for $0\leq n\leq M$ and $T(0)=1$.
Therefore, there exist positive constants $T_0$ and $T_1$ such that
\begin{equation}
\label{conditionT}
 T_1\leq T(n) \leq T_0,~ |T'(n)|\leq T_0\ \text{for} \ 0\leq n\leq M. 
\end{equation}

It is noted that if $\mathbf{v}=(-x_2\sqrt{L}/\eta^2, x_1\sqrt{L}/\eta^2,0)$ and $n=n_{b,\rho}$, then entropy function $\eqref{FEP}_3$  is satisfied.

The assumptions will be discussed in \cref{rem-assumptions}. The main results of this paper can be stated as:
\begin{thm}
\label{thm1}
Assume that $\rho$ minimizes $E_b$ in $W^b$. Then $\rho\in C(\mathbb{R}^3)$. Moreover, $\rho$ is continuously differentiable where it is positive, and \eqref{sss} is satisfied there.
\end{thm}
\begin{thm}
\label{thm2}
If (A1)-(A7) are all satisfied, then there is a $\rho\in W^b$ which minimizes $E_b$. Moreover, there is a constant $R_0$ such that any such $\rho$ vanishes for $|x|\geq R_0$.
\end{thm}
\begin{thm}
\label{thm3}
For any $\xi>1$, if (A1)-(A7) are all satisfied, then there is a $(\rho,n)$ in
\begin{equation}
\label{W}
W=\{(\rho,n) \ | \ \text{there exists some} \ b\in [\frac{1}{\xi},\xi] \ \text{s.t.}\ \rho\in W^b, n=n_{b,\rho} \}.
\end{equation}
which minimizes $E(\rho,n)$. 
Moreover, there is a constant $R_0$ such that any such $\rho$ vanishes for $|x|\geq R_0$.
\end{thm}

\vspace{0.5cm}
\cref{thm1} implies that the energy minimizers of $E_b$ are the non-isentropic axisymmetric rotating solutions to the full Euler-Poisson equations \eqref{FEP}.
 \cref{thm2} and \cref{thm3} show that such minimizers exist in $W_b$ and $W$.
	Here are some remarks on the results and the conditions (A1)-(A7).
\begin{remark}\label{rem-symmetry}
		Our result extends that of the isentropic case (\cite{Auchmuty1971}) to the non-isentropic case.
		However, in order to make sure that the energy minimizers are the solutions to \eqref{sss}, 
		here $\rho$ has to be assumed to satisfy $\rho(x)=\rho(r_b(x))$
		(see \eqref{symmetry} in the proof of \cref{thm1}).
		In this regime, the definition of $r_b(x)$ determines the symmetries of the solutions.
		
	    Of course we can define $r_b(x)$ as any other kinds of hypersurfaces. However, if we take $r_b(x)$ as $\eta(x)$, then $n_{b,\rho}(r_b(x))$ is same as $m_{\rho}(\eta(x))$. In this way, the entropy of the star is the constant in any hypersurface $\{\eta(x)=a\}_{a>0}$, which contracts to the observations in astrophysics: the substance of the star has higher temperature or entropy as it approaches to the core.
		
		This ellipsoidally symmetric condition ($\rho(x)=\rho(r_b(x))$) in this paper is stronger than the axisymmetric one ( $\rho(x)=\rho(\eta(x),z(x))$ ) (\cite{Auchmuty1971}), but weaker than the spherical symmetry (\cite{Luo2004}):
		when $b=1$, $\rho(x)=\rho(r_b(x))$ is exactly the spherical symmetry, and  \cref{thm3} implies that the energy minimizers exist among a large class of ellipsoidally symmetric functions.
		The motivation to consider this ellipsoidal symmetry is a compromise.	
		Actually, if the star is rotating with non-zero uniform angular velocity, the  configuration is not likely to be spherically symmetric, see \cite{Caffarelli1980}.
		We prefer to consider solutions rather than spherically symmetric ones, but as explained in the above, in this regime we cannot define $r_b(x)$ just in axisymmetric motions.	
\end{remark}
\begin{remark}
		The main difference between our result and the previous result (\cite{Wu2015}) of the existence of variational non-isentropic solutions is that, the original constraint $\int_{\mathbb{R}^3} \rho=M$ in \cite{Auchmuty1971} is kept unchanged here, while in \cite{Wu2015} it was changed  to $\int_{\mathbb{R}^3}f_1w=P$, where  $w=\frac{\gamma}{\gamma-1}e^{\frac{\gamma-1}{\gamma}S}\rho^{\gamma-1}$ and $f_1$ denotes a function of the angular velocity.
		We prefer to keep this more physical constraint, since in this way the energy minimizers can be interpreted as steady states of a fixed mass of rotating gaseous star with the prescribed angular momentum.
\end{remark}
\begin{remark} \label{rem-assumptions} ~\\
\begin{enumerate}  
	\item[1)] We adopt the assumptions in \cite{Auchmuty1971} as the condition (A1) and (A3). For polytropic gases, $p=e^S\rho^{\gamma}$, the condition (A2) is automatically satisfied and the condition (A1) just implies $\gamma>\frac{4}{3}$.
	\item[2)] The condition (A4) is also imposed by Luo and Smoller in \cite{Luo2008, Luo2009}. The monotonicity condition of (A4) is called the S\"{o}lberg stability criterion(\cite{Tassoul2015}).
	\item[3)] The condition (A6) is set such that for any $0\leq a\leq1, 0\leq n\leq M$,
	\begin{equation}
	\label{conditionS}
	e^{S(an)}\geq a^{2/3} e^{S(n)}, ~\text{and}~ e^{S(M-aM+an)}\geq a^{2/3}e^{S(n)}.
	\end{equation}	
	Indeed, 
	\begin{equation*}
	\begin{aligned}
	&\exp\Big\{ S(an)-S(n)\Big\}
		\geq \exp\Big\{-\sup\limits_{n\in [0,M]}|S'(n)| \cdot (1-a)n\Big\}\\
		&\qquad \geq \exp\Big\{-\frac{2}{3M}(1-a) n  \Big\} 
		\geq \exp\Big\{-\frac{2}{3}  (1-a) \Big\}\\
	&\qquad \geq \exp \Big\{\frac{2}{3}\ln a \cdot \frac{1-a}{-\ln a} \Big\}
		\geq \exp\Big\{\frac{2}{3}  \ln a \Big\}  =  a^{2/3}\\
	&\exp\Big\{ S(M-aM+an)-S(n)\Big\}
		\geq \exp\Big\{-\sup\limits_{n\in [0,M]}|S'(n)| \cdot (1-a)(M-n)\Big\}\\
	&\qquad \geq \exp\Big\{-\frac{2}{3M} (1-a)(M-n) \Big\} 
		\geq  a^{2/3}
		\end{aligned}
		\end{equation*} 
	Here the inequality $\frac{1-a}{-\ln a}\leq 1$ is used ($\frac{-\ln a}{1-a}$ is monotone decreasing on $(0,1)$ and $\lim\limits_{t\rightarrow 1^- } -\frac{\ln t}{1-t}=1$).
	So condition \eqref{conditionS} is weaker than the one in \cite{Luo2004}, under which $|S'|$ is required to be sufficiently small. 
    \item[4)] The condition (A7) requires that the entropy of the gaseous star is decreasing when sufficiently near the vacuum. 
	This condition is weaker than those conditions on the entropy in \cite{Luo2004, Wu2015}
\end{enumerate}
\end{remark} 

Throughout the paper, we will write $C$ for a generic constant and 
$$
\int f\ dx=\int_{\mathbb{R}^3}f(x)\ dx.
$$
$\|\cdot\|_p$ denotes the norm in the $L^p$ function space:
$$
\|f\|_p:=\left(\int |f|^p \ d\mu \right)^{1/p}.
$$

The rest of the paper is arranged as the following. 
We give the basic inequalities used in this paper and prove 
 \cref{thm1} in Section 3. 
And then by analyzing an energy-minimizing sequence, the existence of the energy minimizers on $W^b$ is established in Section 4, 
which is exactly  \cref{thm2}. 
This existence theorem is improved by enlarging the class $W^b$ of $\rho$ to $\bigcup\limits_{b\in [\frac{1}{\xi},\xi]} W^b$ in Section 5.

\section{Inequalities and Regularity of minimizers}

\begin{prop}
\label{prop1}
Assume that $\rho\in L^1\cap L^p$. If $1<p\leq \frac{3}{2}$. Then $B\rho\in L^r$ for $3<r<3p/(3-2p)$, and 
\begin{equation} \label{ineq1}
\| B\rho\|_r\leq C(\|\rho\|_1^{\alpha_1}\|\rho\|_p^{1-\alpha_1}+\|\rho\|_1^{\alpha_2}\|\rho\|_p^{1-\alpha_2})
\end{equation}
where $0<\alpha_1,\alpha_2<1$.\\
If $p>\frac{3}{2}$, then $B\rho$ is bounded and continuous and satisfies \eqref{ineq1} with $r=\infty$.
\end{prop}

\begin{prop}
\label{prop2}
Suppose that $\rho \in L^1\cap L^{4/3}$. Then 
\begin{equation}
\label{ineq2}
\|\nabla B\rho\|_{2}^2=|\int\rho B\rho|\leq C\int |\rho|^{4/3}(\int|\rho|)^{2/3}.
\end{equation}
\end{prop}

\begin{prop}
\label{prop3}
If $\rho \in L^1\cap L^p$ for some $p>3$, then $B\rho$ is continuously differentiable.
\end{prop}

The above three propositions are proved in \cite{Auchmuty1971} by standard H\"{o}lder inequality and estimates for the Riesz potential.

Let $B_R=\{x\in \mathbb{R}^3| |x|<R \}$. 
\begin{prop}
\label{prop4}
Let $\{\rho_j\}_{j\in \mathbb{N}^*}$ be a bounded sequence in $L^{p}(\mathbb{R}^3)$, $p>3$.
Suppose that 
$$
\rho_j\rightharpoonup \rho_0 \quad \text{weakly in} \ L^{p}(\mathbb{R}^3),
$$
and $\int \rho_j=M$, $j\in \mathbb{N}^*\cup \{0\}$.
Then for any $R>0$, 
$$
\nabla B(\chi_{B_R}\rho_j)\rightarrow \nabla B(\chi_{B_R}\rho_0) \quad \text{strongly in }\ L^2(\mathbb{R}^3),
$$
where $\chi$ is the characteristic function. 
\end{prop}
\begin{proof}
See \cite[Lemma~3.7]{Rein2001} and \cite[Lemma~3.7]{Luo2008}.
\end{proof}

\subsection{Variational Solutions}
Given $\rho\in W^b$, choose $\epsilon>0$ small enough such that $\Omega_{\epsilon}(\rho)=\{x| \ \epsilon<  \rho(x)<\epsilon^{-1}\}$ has positive measure. Define
\begin{equation}
\label{P_0}
\begin{aligned}
P_{\epsilon}(\rho)&= \{\sigma:\ \mathbb{R}^3\rightarrow\mathbb{R}|\ \sigma\in L^{\infty}(\mathbb{R}^3);\ \sigma(x)=\sigma(r_b(x)) ; \\
&\qquad \sigma=0 \ \text{a.e.} \ \text{outside}\ \Big(\Omega_\epsilon(\rho) \cup \{x|\ \rho(x)=0 \}\Big); \\
&\qquad \sigma(x)\geq 0 \ \text{if} \ \rho(x)=0;\ \sigma(x)=0\ \text{if}\ \eta(x)<\epsilon \ \text{or}\ |x|>\epsilon^{-1} \},  \\  
P_0(\rho)&=\bigcup\limits_{\epsilon>0}P_{\epsilon}(\rho). 
\end{aligned}
\end{equation}

\begin{lem}
\label{lem1}
Suppose that $\rho\in W^b$ and $P_0(\rho)$ is defined as above. Then 
\begin{equation}
\lim\limits_{t\rightarrow0^+}\frac{E_b(\rho+t\sigma)-E_b(\sigma)}{t}=\int E'_b(\rho(x))\sigma(x) \ dx \quad \text{for} \ \forall\ \sigma \in P_0(\rho) .
\end{equation}
\end{lem}
\begin{proof}
Since the last two integrals in $\widehat{E}_1(\rho)$ and $E_b(\rho)$ are the same, 
the proof of \cite[Lemma~1]{Auchmuty1971} implies that
it remains to prove that 
for any $ \epsilon >0$ and any $ \sigma \in P_{\epsilon}(\rho)$,
\begin{equation}
\label{claim1}
\begin{aligned}
\lim\limits_{t\rightarrow0^+}& t^{-1}\Big[ \int A(\rho+t\sigma)T(n_{b,\rho+t\sigma}(r_b))-\int A(\rho)T(n_{b,\rho}(r_b))\Big]=\\
& \int \Big[A'(\rho(x))T(n_{b,\rho}(r_b(x)))+\int_{r_b(y)>r_b(x)}A(\rho(y))T'(n_{b,\rho}(r_b(y)))\ dy \Big]\sigma(x)\ dx
\end{aligned}
\end{equation}

Actually, noting that $A$ and $T$ are both continuously differentiable, it follows from the Mean Value Theorem that
\begin{align*}
I&:=t^{-1}\Big[ \int A(\rho+t\sigma)T(n_{b,\rho+t\sigma}(r_b))-\int A(\rho)T(b,n_\rho(r_b))\Big]\\
&=\int \left[ \frac{A(\rho+t\sigma)-A(\rho)}{t}T(n_{b,\rho+t\sigma}(r_b)) +A(\rho)\frac{T(n_{b,\rho+t\sigma}(r_b))-T(n_{b,\rho}(r_b))}{t} \right]\\
&=\int A'(\rho+\theta \sigma)\sigma T(n_{b,\rho+t\sigma}(r_b))+\int A(\rho) T'(n_{b,\rho+\theta'\sigma}(r_b))n_{b,\sigma}(r_b),
\end{align*}
where $\theta$, $\theta'$ are functions of $x$, and $0<\theta,\theta'<t$. 
Since $A'$ is stricty increasing,  $|\sigma(x)|\leq k$ for some constant $k$, and $\sigma(x)=0$ if $\rho(x)>\epsilon^{-1}$ or $|x|>\epsilon^{-1}$, with $T$ satisfying the condition \eqref{conditionT}, it holds that 
\begin{align*}
|A'(\rho+\theta \sigma)\sigma T(b,n_{\rho+t\sigma}(r_b))|\leq A'(\epsilon^{-1}+t\|\sigma\|_{\infty})\|\sigma\|_{\infty}T_0\\
|A(\rho) T'(n_{b,\rho+\theta'\sigma}(r_b))n_{b,\sigma}(r_b)|\leq T_0M A(\rho).
\end{align*}
Since $\sigma$ is compactly supported, $A'(\epsilon^{-1}+t\|\sigma\|_{\infty})\|\sigma\|_{\infty}T_0$ is integrable. Thus, it follows from the Lebesgue's Dominated Convergence Theorem that
$$I\rightarrow \int A'(\rho)T(n_{b,\rho}(r_b))\sigma+\int A(\rho)T'(n_{b,\rho}(r_b))n_{b,\sigma}(r_b)=: I_1+I_2 \quad \text{as}\ t\rightarrow 0^+.$$

Let 
$$\tilde{\rho}(s)=\int_{r_b(y)=s} \rho(y) \ dS_y,$$
and then the fact that $\rho(x)=\rho(r_b(x))$ and $\sigma(x)=\sigma(r_b(x))$ implies the following equality
\begin{equation*}
\begin{aligned}
I_2&=\int_{0}^{\infty}\widetilde{A(\rho)}(r_b)T'(n_{b,\rho}(r_b))\int_{0}^{r_b}\tilde{\sigma}(s) \ ds \ dr_b\\
&=\int_{0}^{\infty}\tilde{\sigma}(s)\int_{s}^{\infty}\widetilde{A(\rho)}(r_b)T'(n_{b,\rho}(r_b))\ dr_b\ ds\\
&=\int \Big[\int_{r_b(y)>r_b(x)}A(\rho(y))T'\Big(n_{b,\rho}(r_b(y))\Big)\ dy \Big]\sigma(x) \ dx.
\end{aligned}
\end{equation*}
Hence, \eqref{claim1} is verified.

Therefore, this lemma is proved.
\end{proof}

\begin{lem}
\label{lem2}
Suppose that $\rho$ minimizes $E_b$ on $W^b$. Then there is a constant $\lambda$ such that
\begin{eqnarray}
&& E'_b(\rho)\geq \lambda \qquad \text{a.e.}, \label{eprime1}\\
&& E'_b(\rho)=\lambda \qquad \text{a.e.}, \ \text{where}\ \rho(x)>0. \label{eprime2}
\end{eqnarray}
\end{lem}
\begin{proof}
See \cite[Lemma~2]{Auchmuty1971}
\end{proof}

 \cref{lem1} and  \cref{lem2} show the Euler-Lagrange equations of the minimizers of $E_b$ on $W^b$.
 \cref{lem3} will provide an inequality that is the key to run the bootstrap arguments to improve the regularities of the minimizers of $E_b$ on $W^b$.

\begin{lem}
\label{lem3}
Suppose that $\rho$ minimizes $E_b$ on $W^b$. Then there exist constants $c_1$ and $K$ such that 
\begin{equation}
\label{ineq3}
\rho^{1/3}\leq c_1 B\rho
\end{equation}
a.e. on the set where $\rho\geq K$. 
\end{lem}
\begin{proof}
It follows from the definition of $A$ and the assumptions on $f$ (A1) that 
for any $c_2$, $A'(s)\geq c_2 s^{1/3}$ for large $s$. 
By \eqref{eprime2} and the condition of $T$ \eqref{conditionT}, when $\rho>0$,
it holds that
\begin{align*}
& T_1A'(\rho)\leq A'(\rho)T(n_{b,\rho}(r_b))\\
\leq& -\int_{r_b(y)>r_b(x)}A(\rho(y))T'(n_{b,\rho}(r_b(y))) dy+ B\rho+\lambda\\
\leq& T_0 \int A(\rho(y))\ dy+B\rho+\lambda.
\end{align*}
Then inequality \eqref{ineq3} follows by choosing $c_1>c_2$ and $K$ large enough.
\end{proof}
\begin{remark}
The constants $K$ and $c_1$ in  \cref{lem3} depend on $\rho$.
\end{remark}

\begin{proof}[\textbf{Proof of  \cref{thm1}}]
Since $T$ is continuously differentiable, then it follows from the proof of \cite[Theorem~1]{Auchmuty1971} that 
$\rho\in C(\mathbb{R}^3)$ and is continuous differentiable where it is positive. 
When $\rho$ is positive, taking the gradient in \eqref{eprime2} gives that
\begin{equation*}
\begin{aligned}
& A''(\rho)T(n_{b,\rho}(r_b))\nabla \rho+A'(\rho)T'(n_{b,\rho}(r_b))\left(\int_{r_b(y)=r_b}\rho\right)\nabla r_b\\
&\qquad -\left(\int_{r_b(y)=r_b}A(\rho)
T'(n_{b,\rho}(r_b))\right)\nabla r_b=\cdots+\nabla B\rho,
\end{aligned}
\end{equation*}
where ``$\cdots$'' are terms obtained from the third term in \eqref{eprime}, which can be handled in the same way as that in \cite{Auchmuty1971} and thus are omitted here.
Note that $\rho(x)=\rho(r_b(x))$ for any $\rho\in W^b$, and hence 
\begin{equation}
\label{symmetry}
\frac{\int_{r_b(y)=r_b}A(\rho)}{\int_{r_b(y)=r_b}\rho}=\frac{A(\rho)}{\rho}.
\end{equation}
By the definition of $A$ \eqref{A}, $A''(\rho)=\frac{f'(\rho)}{\rho}$ and $A'(\rho)-\frac{A(\rho)}{\rho}=\frac{f(\rho)}{\rho}$.
So it follows that 
\begin{equation*}
\frac{f'(\rho)}{\rho}T(n_{b,\rho}(r_b))\nabla\rho+\frac{f(\rho)}{\rho}T'(n_{b,\rho}(r_b))\left(\int_{r_b(y)=r_b}\rho\right)\nabla r_b=\cdots+\nabla B\rho
\end{equation*}
Multiplying it with $\rho$ gives \eqref{sss}.
Therefore, the proof is finished.
\end{proof}

\section{Existence of Rotating Solutions} 
In this section, the existence of the compactly supported minimizers of the energy functional $E_b$ on $W^b$ is established.
We first prove the existence of the minimizers on $W_{R}^b$  \eqref{W,R}.
Then it should be pointed out that the constrained minimizers $\rho_R$ on $W_{R}^b$ actually form a subsequence $\{\rho_{R_j}\}_{j\in\mathbb{N}^*}$ 
which minimize the energy functional $E_b$ on $W^b$ and has a weak limit $\rho_0$ in $L^{4/3}(\mathbb{R}^3)$.
At last, the uniform estimates of $\{\rho_{R_j}\}_{j\in\mathbb{N}^*}$ show that $\rho_0$ is the compactly supported energy minimizer on $W^b$.

\subsection{Existence of Constrained Minimizers}
Given $R>0$, denote
\begin{equation}
\label{W,R}
\begin{aligned}
W^b_{R}&=\{ \rho: \mathbb{R}^3\rightarrow \mathbb{R} | \; \rho(x)\geq0\ \text{a.e.};\ \int\rho(x)\ dx=M;\ \rho(x)=\rho(r_b(x));\\
& \
\rho(x)=0 \ \text{if}\ |x|\geq R;\ \rho(x) \leq R \ \text{a.e.}    \}.
\end{aligned}
\end{equation}

\begin{lem}
\label{lem4}
For any $R$ large enough, there exists a $\rho_R\in W^b_{R}$ which minimizes $E_b$ on $W^b_{R}$. And such $\rho_R$ is continuous if $|x|<R$. Furthermore, there exists a constant $\lambda_R$ such that 
\begin{eqnarray}
&& E_b'(\rho_R)\geq \lambda_R \qquad \text{a.e.},\ \text{where}\ |x|<R,\ \rho_R(x)<R \label{eprime1r}\\
&& E_b'(\rho_R)=\lambda_R \qquad \text{a.e.}, \ \text{where}\ |x|<R,\ 0<\rho_R(x)<R. \label{eprime2r}
\end{eqnarray}
\end{lem}
\begin{proof}
Since the last two integrals in $\widehat{E}_1(\rho)$ and $E_b(\rho)$ are the same, 
the proof of \cite[Theorem~$B_R$]{Auchmuty1971} implies that it suffices to prove that $\int A(\rho) T(n_{b,\rho})$ is weakly lower semi-continuous in $W^b_{R}$. Here $W^b_{R}$ is a bounded, closed and convex subset of $L^p$, $1< p\leq \infty$ and thus it is also weakly closed and weakly compact(see \cite[Theorem~6,~Chapter~10]{Lax}) .

Suppose that $\{\rho_n\}_{n=1}^{\infty}, \rho_0 \in W^b_{ R}$ and $\rho_n$ converges weakly to $\rho_0$ in $L^{p}$.
Since $\{\rho\in W^b_{ R} | \int A(\rho)T(n_{b,\rho_0}) \leq k \}$ is convex and closed, 
then it is weakly closed . 
Therefore, $\liminf\limits_{n\rightarrow \infty}\int A(\rho_n)T(n_{b,\rho_0})\geq\int A(\rho_0)T(n_{b,\rho_0})$.

Since $n_{b,\rho_n}\rightarrow n_{b, \rho_0}$ a.e. and these functions are continuous, monotonically increasing in $\mathbb{R}$ and bounded uniformly, 
then $n_{b,\rho_n}\rightarrow n_{b, \rho_0}$ uniformly on compact subsets of $[0,\infty)$ (see \cite[Problem~13,~p.167]{Rudin}).
Therefore, $T(n_{b, \rho_n})\rightarrow T(n_{b, \rho_0})$ uniformly.

So for any $\epsilon >0$ , there exists $N(\epsilon)$ such that $\forall n>N$ 
\begin{equation*}
\begin{aligned}
&\inf\limits_{j>n} \int A(\rho_j)T(\rho_0) +\frac{\epsilon}{2}\geq \int A(\rho_0)T(\rho_0),\\
&\left| \int A(\rho_n)T(\rho_n)- \int A(\rho_n)T(\rho_0) \right|
\leq A(R) \int |T(\rho_n)-T(\rho_0)|<\frac{\epsilon}{2}.
\end{aligned}
\end{equation*}
Then 
\begin{equation*}
\begin{aligned}
\inf\limits_{j>n} \int A(\rho_j)T(\rho_n) &\geq \inf\limits_{j>n} \int A(\rho_j)T(\rho_0) -\sup\limits_{j>n} \left| \int A(\rho_j)T(\rho_0)- \int A(\rho_j)T(\rho_j) \right|\\
&\geq A(\rho_0)T(\rho_0)-\epsilon.
\end{aligned}
\end{equation*}
This inequality gives $\liminf\limits_{n\rightarrow \infty}\int A(\rho_n)T(n_{b,\rho_n})\geq\int A(\rho_0)T(n_{b,\rho_0})$.
Hence a minimizer $\rho_R$ on $W_{R}^b$ exists. 
The rest of the proof can be completed by following  \cref{lem2} and the proof of  \cref{thm1}.
\end{proof}

 \cref{lem4} corresponds to  \cref{lem1},  \cref{lem2} and  \cref{lem3} for the minimizers of $E_b$ on $W^b_{R}$.
The Euler-Lagrange equations provided by  \cref{lem4} will be used throughout the remaining proof of the paper.
\begin{lem}
\label{lem5}
There exist $k_0>0$ and $\pi_1>0$ such that $\|\rho_{R}\|_{4/3}\leq k_0$, when $R\geq\pi_1$. 
\end{lem}
\begin{proof}
For any fixed $\pi_0>0$, we have $E_b(\rho_R)\leq E_b(\rho_{\pi_0})$, when $R\geq \pi_0$. 
\eqref{ineq2} and the condition of $T$ \eqref{conditionT} imply that
\begin{equation}
\label{ineq4}
T_1\int A(\rho_R)\leq E_b(\rho_{\pi_0})+C\int \rho_R^{4/3}.
\end{equation}
With the inequality, the rest of the proof follows easily from \cite[Lemma~4]{Auchmuty1971}
\end{proof}
 \cref{lem5} shows that $\rho_{R}$ has a uniform bound in $L^{4/3}$ norm.

\subsection{Uniform estimates of the constrained minimizers}
Since $\rho_R$ are uniformly bounded in $L^{4/3}(\mathbb{R}^3)$,
there exist a subsequence of $\rho_R$, i.e. $\{\rho_{R_j} \}_{j\in \mathbb{N}^*}\subset 
W^b$  and $\rho_0\in L^{4/3}$ such that 
\begin{equation}
\label{sequence}
\rho_{R_j} \rightharpoonup\rho_0 \ \text{weakly in } \ L^{4/3}.
\end{equation}
Since $W^b$ is a convex set, $\rho_0$ also satisfies $\rho_0(x)=\rho_0(r_b(x))$. (See \cite[Theorem~6,~Chapter~10]{Lax}).
The remaining part of this section is devoted to deriving the uniform estimates of their $L^{\infty}$ norm and radii.
We would not focus on the subsequence until Lem \ref{lem12}.

\begin{lem}
\label{lem6}
There exists $\pi_2>\pi_1$ such that $\lambda_R<C E_b(\rho_{\pi_0})+C k_0^{4/3}$ for all $R \geq \pi_2$.
\end{lem}
\begin{proof}
Since $T_1 \leq |T(n)| \leq T_0, |T'(n)|\leq T_0$, the proof of \cite[Lemma~5]{Auchmuty1971} implies that 
$$
\lambda_R\leq o(R^{-1})+O(R^{-2})-O(R^{-1})
+\int_{r_b(y)>r_b(x)}A(\rho(y))T'(n_{b,\rho}(r_b(y)))\ dy.
$$
on a special subset of $B_R$. 
So there exists $\pi_2>\pi_1$ such that for any $R\geq \pi_2$, 
\begin{equation*}
\begin{aligned}
\lambda_R &\leq 
\sup\left|\int_{r_b(y)>r_b(x)}A(\rho(y))T'(n_{b,\rho}(r_b(y)))\ dy \right|\\
&\leq T_0\int A(\rho_R) \overset{\eqref{ineq4}}{\leq} C E_b(\rho_{\pi_0})+C k_0^{4/3}.
\end{aligned}
\end{equation*}
Hence this lemma is proved.
\end{proof}

\begin{lem}
\label{lem7}
$\|\rho_R\|_{\infty}\leq k_1$, all $R$.
\end{lem}
\begin{proof}
It follows from \eqref{eprime2r} and  \cref{lem6} that, 
\begin{equation*}
\begin{aligned}
T_1A'(\rho_R)&\leq \max\{0,\lambda_R-\int_{r_b(y)>r_b(x)}A(\rho_R(y))T'(n_{b,\rho_R}(r_b(y)))\ dy+B\rho_R \}\\
&\leq  \max\{0,CE_b(\rho_{\pi_0})+Ck_0^{4/3}+B\rho_R \},\qquad \text{for}\ R\geq\pi_2
\end{aligned}
\end{equation*}
So as  \cref{lem3}, there exist constants $C_1'$, $K'$ such that $\rho^{\gamma-1}\leq C_1'B\rho$ a.e. when $\rho_R\geq K'$. 
Since $CE_b(\rho_1)+Ck_0^{4/3}$ does not depend on $\rho_R$, $C_1'$ and $K'$ do not depend on $\rho_R$ either. 
Thus in the proof of  \cref{thm1}, $\| \rho_R\|_{\infty}$ can be shown to be uniformly bounded when $R\geq \pi_2$. 
Since $\rho_R\in W^b_{R}$, $\| \rho_R\|_{\infty}\leq \pi_2$ for $R\leq \pi_2$. So this lemma is proved.
\end{proof}
 \cref{lem7} implies that 
$\rho_R$ could not concentrate at some point as $R\rightarrow \infty$.

The proof of the following three lemmas are standard, see \cite[Lemma~6,~Corollary,~Lemma~7]{Auchmuty1971}. 
$T$ would not change the result for $T$ is bounded.

\begin{lem}
\label{lem8}
There exists an $e<0$ such that $E_b(\rho_R)\leq e$ for all large $R$.
\end{lem}

\begin{lem}
\label{lem9}
$\|B\rho_R\|_{\infty}\geq -2e/M$ for all large $R$. 
\end{lem}

\begin{lem}
\label{lem10}
There is an $\epsilon_0>0$ such that for each $R$ 
$$
\int_{|x-x_R|<1}\rho_R(x)\ dx\geq \epsilon_0
$$
for some $x_R$.
\end{lem}

\begin{lem}
\label{lem11}
There is a constant $r_0$ such that if $x_R$ is as in \cref{lem10}, then $|x_R|+1\leq r_0$, all $R$.
\end{lem}
\begin{proof}
Claim that when $r_b(x_R)>\max\{b^{-1},1\}+1$,
\begin{equation*}
\begin{aligned}
&\int_{r_b(x)<2r_b(x_R)}\rho_R(x)\ dx \geq  \frac{\min\{b,1\}r_b(x_R)\epsilon_0}{4\pi}
\end{aligned}
\end{equation*}
Once it is proved, since $\rho_R$  is of total mass $M$, it follows that 
\begin{equation*}
\begin{aligned}
\min\{b^{-1},1\}|x_R| & \leq r_b(x_R) \leq  \left(\frac{4\pi M }{\epsilon_0}+1\right) (\max\{ b^{-1},1\}+1)\\
 |x_R| &\leq  \left(\frac{4\pi M }{\epsilon_0}+1\right) (2b+b^{-1}+1)
\end{aligned}
\end{equation*}
So it suffices to prove the above claim.

Suppose that $$x=(r_b\cos\theta\cos\beta,r_b\sin\theta\cos\beta,br_b\sin\beta),$$
where $r_b$ and $\theta$ are defined as before, and $\beta\in[-\frac{\pi}{2},\frac{\pi}{2}]$ satisfies $\sin\beta=z(x)b^{-1}r_b(x)^{-1}$ and $\cos\theta=\eta(x)r_b(x)^{-1}$.
And define $r_b(x_R)$, $\theta(x_R)$ and $\beta(x_R)$ the same way as above for $x_R$.
Then $|x-x_R|<1$ implies that
\begin{equation}
\label{ineq6}
\begin{aligned}
&[r_b\cos\theta\cos\beta-r_b(x_R)\cos\theta(x_R)\cos\beta(x_R)]^2+[r_b\sin\theta\cos\beta
-\\&\qquad r_b(x_R)\sin\theta(x_R)\cos\beta(x_R)]^2
+b^2[r_b\sin\beta-r_b(x_R)\sin\beta(x_R)]^2<1.
\end{aligned}
\end{equation}
\begin{enumerate}
\item[(1)] If $b\leq1$, inequality \eqref{ineq6} gives 
\begin{equation*}
\begin{aligned}
&[r_b\cos\theta\cos\beta-r_b(x_R)\cos\theta(x_R)\cos\beta(x_R)]^2+[r_b\sin\theta\cos\beta
-\\&\qquad r_b(x_R)\sin\theta(x_R)\cos\beta(x_R)]^2
+[r_b\sin\beta-r_b(x_R)\sin\beta(x_R)]^2<b^{-2}.
\end{aligned}
\end{equation*}
\begin{equation*}
r_b^2+r_b(x_R)^2-2r_br_b(x_R)\{\cos\beta\cos\beta(x_R)\cos[\theta-\theta(x_R)]+\sin\beta\sin\beta(x_R)\}<b^{-2}.
\end{equation*}
Since $\beta\in[-\frac{\pi}{2},\frac{\pi}{2}]$, it holds that
\begin{equation*}
r_b^2+r_b(x_R)^2-2r_br_b(x_R)[\cos\beta\cos\beta(x_R)+\sin\beta\sin\beta(x_R)]<b^{-2}.
\end{equation*}
Therefore, 
\begin{equation*}
\begin{aligned}
& |r_b-r_b(x_R)|^2<b^{-2},\\
& \cos |\beta-\beta(x_R)|> \frac{r_b^2+r_b(x_R)^2-b^{-2}}{2r_br_b(x_R)}.
\end{aligned}
\end{equation*}
Here  $0=\frac{b^{-2}-b^{-2}}{2r_br_b(x_R)}\leq\frac{r_b^2+r_b(x_R)^2-b^{-2}}{2r_br_b(x_R)}\leq\frac{r_b^2+r_b(x_R)^2-|r_b-r_b(x_R)|^2}{2r_br_b(x_R)}=1$.
So it holds that
when $r_b(x_R)>\max\{b^{-1},1\}+1$,
\begin{equation*}
\sin|\beta-\beta(x_R)|<\sqrt{\frac{b^{-2}-[r_b-r_b(x_R)]^2}{2r_br_b(x_R)}}<\sqrt{\frac{b^{-2}}{2[r_b(x_R)-1]r_b(x_R)}}<\frac{1}{br_b(x_R)}.
\end{equation*}
Thus, \cref{lem10} yields that 
\begin{align*}
\epsilon_0&\leq \iiint\limits_{\substack{|r_b-r_b(x_R)|\leq b^{-1},\\ |\beta-\beta(x_R)|\leq 2b^{-1}[r_b(x_R)]^{-1}}} \rho(y) dy
\\
&\leq
\int_{r_b(x_R)-b^{-1}}^{r_b(x_R)+b^{-1}}\int_{0}^{2\pi}\int_{\beta(x_R)-|\beta-\beta(x_R)|}^{\beta(x_R)+|\beta-\beta(x_R)|} \rho(r_b) \ br_b^2\cos\beta\ d\beta d\theta dr_b\\
&\leq 4\pi \cos [\beta(x_R)] \sin |\beta-\beta(x_R)|\int_{r_b(x_R)-b^{-1}}^{r_b(x_R)+b^{-1}}\rho(r_b)\  br_b^2\ dr_b\\
&\leq \frac{4\pi}{br_b(x_R)}\int_{r_b(x_R)-b^{-1}}^{r_b(x_R)+b^{-1}}\rho(r_b)\  br_b^2\ dr_b,
\end{align*}
and the mass in the open set  $\{x|r_b(x)<2r_b(x_R) \}$ satisfies
\begin{align*}
\iiint\limits_{r_b(y)<2r_b(x_R)} \rho(y)\ dy\geq
\int_{r_b(x_R)-b^{-1}}^{r_b(x_R)+b^{-1}}\rho(r_b) \ br_b^2  \  dr_b\geq\frac{br_b(x_R)\epsilon_0}{4\pi}.
\end{align*}
So the claim is proved when $b\leq 1$. 

\item[(2)] If $b\geq1$, then the inequality \eqref{ineq6} implies that 
\begin{equation*}
\begin{aligned}
&[r_b\cos\theta\cos\beta-r_b(x_R)\cos\theta(x_R)\cos\beta(x_R)]^2+[r_b\sin\theta\cos\beta
-\\&\qquad r_b(x_R)\sin\theta(x_R)\cos\beta(x_R)]^2
+[r_b\sin\beta-r_b(x_R)\sin\beta(x_R)]^2<1.
\end{aligned}.
\end{equation*}
Then remaining of the estimate is the same.
\end{enumerate}
Thus the claim is proved and  \cref{lem11} follows.
\end{proof}

\vspace{0.5cm}

Up to now, we have just used conditions (A1), (A3) and (A5). 
In the following, conditions (A2), (A4), (A6) and (A7) are used
and we will focus on the subsequence $\{\rho_{R_j}\}$
\begin{lem}
\label{lem12}
$\{\rho_{R_j}  \}_{j\geq1}$ is the minimizing sequence of $E_b$ on $W^b$, i.e. 
$$
\lim\limits_{j\rightarrow \infty}E_b(\rho_{R_j})=\min\limits_{\rho\in W^b}E_{b}(\rho).
$$
\end{lem}
\begin{proof}
$E_b(\rho_{R_j})$ is a decreasing sequence and by \eqref{ineq2} 
$$
E_b(\rho_{R_j})
\geq -\int \rho_{R_j} B\rho_{R_j} \geq -CM^{2/3}\|\rho_{R_j}\|^{4/3}_{4/3}\geq -CM^{2/3}k_0.
$$
So $\lim\limits_{j\rightarrow\infty}E_b(\rho_{R_j})$ exists,
and it suffices to prove that for any $\tau\in W^b$, $E_b(\tau)\geq \lim\limits_{j\rightarrow \infty}E_b(\rho_{R_j})$.

Take a nonnegative function $\sigma\in P_{\epsilon}(\tau)$ for some $\epsilon>0$, $\int \sigma=1$ and $\sigma(x)=0$ when $\tau(x)=0$. 
Let 
\begin{align*}
\tau'_{\alpha}(x)&=\begin{cases}
\tau(x) & \text{if}\ |x|<\alpha \ \text{or}\ \tau(x)<\frac{1}{2}\alpha\\
0 & \text{otherwise}
\end{cases},\\
\tau''_{\alpha}(x)&=(M-\int \tau'_{\alpha})\sigma,\\
\tau_{\alpha}&=\tau'_{\alpha}+\tau''_{\alpha}.
\end{align*}
The definition of $P_{\epsilon}(\tau)$ (see \eqref{P_0}) implies that $\sigma\in L^{\infty}(\mathbb{R}^n)$ and $\sigma\neq 0$ only if $\eta(x)\geq\epsilon$, $|x|\leq \epsilon^{-1}$ and $\tau(x)<\epsilon^{-1}$.
Hence, as $\alpha\rightarrow\infty$, if $\sigma(x)=0$, then $\tau_{\alpha}(x)$ increases to $\tau(x)$ a.e.; if $\sigma(x)>0$, then $\tau_{\alpha}(x)\rightarrow\tau(x)$ a.e. and $\tau_{\alpha}(x)$ are uniformly bounded by some constant $D$ (depending on $\|\sigma\|_{L^{\infty}}$ and $\epsilon$).

Note that $n_{b,\tau_{\alpha}}(r_b)\rightarrow n_{b,\tau}(r_b)$ and $n_{b,\tau_{\alpha}}$ are uniformly bounded by $M$, and
\begin{align*}
\int A(\tau_{\alpha})T(n_{b,\tau_{\alpha}}(r_b))
& \leq\int \Big\{A(\tau)T(n_{b,\tau}(r_b))\chi_{\{\sigma=0\}}+A(D)T_0\chi_{\{\sigma>0\}} \Big\}\\
&  \leq\int A(\tau(x))T(n_{b,\tau}(r_b(x)))\ dx+\frac{4}{3}\pi \epsilon^{-3}A(D)T_0,
\end{align*}
where $\chi$ is the characteristic function.
Therefore the Lebesgue's Dominated Convergence Theorem implies that $\int A(\tau_{\alpha})T(n_{b,\tau_{\alpha}}(r_b)) \rightarrow\int A(\tau)T(n_{b,\tau}(r_b))$ as $\alpha\rightarrow \infty$.
Similar arguments can also be applied to the other two parts of $E_b$. 
Noting that $\sigma(x)=0$ when $\eta(x)<\epsilon$, it holds that
\begin{align*}
\int\tau_{\alpha}\frac{L(m_{\tau_{\alpha}}(\eta))}{\eta^2}
&\leq \int\Big\{\tau\frac{L(m_{\tau}(\eta))}{\eta^2}\chi_{\{\eta(x)<\epsilon\}}+\tau\frac{L(M)}{\epsilon^2}\chi_{\{\eta(x)\geq\epsilon,\sigma=0\}}\\
&\qquad\qquad+ D\frac{L(M)}{\epsilon^2}\chi_{\{\eta(x)\geq\epsilon,\sigma>0\}}  \Big\};\\
\int \tau_{\alpha}B\tau_{\alpha}&\leq \int (\tau+D\chi_{\{\sigma>0 \}})B(\tau+D\chi_{\{\sigma>0 \}})\\
&=\frac{1}{4\pi}\|\nabla B(\tau+D\chi_{\{\sigma>0 \}})\|_2^2\\
&\leq \Big( (\int \tau B\tau)^{1/2}+(\int D\chi_{\{\sigma>0 \}} BD\chi_{\{\sigma>0 \}})^{1/2} \Big)^2.
\end{align*}
Therefore it follows from the Lebesgue's Dominated Convergence Theorem that 
$$
\lim\limits_{\alpha\rightarrow\infty}E_b(\tau_{\alpha})=E_b(\tau).
$$
Due to that  $E_b(\tau_{\alpha})\geq E_b(\rho_{\alpha})$ when $\alpha$ is large enough, it holds that $E_b(\tau)\geq\lim\limits_{j\rightarrow \infty}E_b(\rho_{R_j})$. Since $\tau$ is chosen arbitrarily, $\{\rho_{R_j} \}_{j\geq1}$ is the minimizing sequence of $E_b$ on $W^b$. 
\end{proof}

\begin{lem}
\label{lem13}
Suppose that the subsequence $\{\rho_{R_j} \}_{j\geq1}$ and $\rho_0$ are given in \eqref{sequence}. 
If the conditions (A4) and (A6) hold, then for any $\delta>0$, there exists a $r>2r_0>0$ such that  
\begin{equation*}
\int_{|x|\geq r}\rho_{R_j}(x)\ dx\leq \delta, \qquad \text{for} \ \text{all}\ R_j.
\end{equation*}
\end{lem}
\begin{proof}
Take an arbitrary  $\rho \in W^b$.
For any $r$ satisfying $r>2r_0>0$,  define
\begin{equation*}
\rho=\rho\chi_{|x|\leq r}+\rho\chi_{|x|> r}=:\rho_1+\rho_2
\end{equation*}
where $\chi$ is the characteristic function. 
Denote 
$$M_i:= \int \rho_i,~a_i:=(\frac{M_i}{M})^{1/3},~\bar{\rho}_i(\cdot):=\rho_i(a_i\cdot)\qquad\qquad i=1,2.$$
Then $a_1, a_2\in [0,1]$,  $\bar{\rho}_1, \bar{\rho}_2 \in W^b$, and
$$m_{\rho_i}(\eta)=a_i^3m_{\bar{\rho}_i}(a_i^{-1}\eta),~ n_{b,\rho_i}(r_b)=a_i^3n_{b,\bar{\rho}_i}(a_i^{-1}r_b)$$
The condition (A6) implies that
\begin{align*}
\int A(\rho)T(n_{b,\rho}(r_b))
&=\int A(\rho_1)T(n_{b,\rho_1}(r_b)) + \int A(\rho_2)T(M-M_2+n_{b,\rho_2}(r_b))\\
&=\int A\Big(\bar{\rho}_1(a_1^{-1}\cdot)\Big)T\Big(a_1^3n_{b, \bar{\rho}_1}(r_b(a_1^{-1}\cdot))\Big) \\
&\qquad+ \int A\Big(\bar{\rho}_2(a_2^{-1}\cdot)\Big)T\Big(M-a_2^3 M+a_2^3n_{b, \bar{\rho}_2}(r_b(a_2^{-1}\cdot))\Big)\\
&\geq\int A\Big(\bar{\rho}_1(a_1^{-1}\cdot)\Big)a_1^2T\Big(n_{b, \bar{\rho}_1}(r_b(a_1^{-1}\cdot))\Big) \\
&\qquad+ \int A\Big(\bar{\rho}_2(a_2^{-1}\cdot)\Big)a_2^2T\Big(n_{b, \bar{\rho}_a}(r_b(a_2^{-1}\cdot))\Big)\\
&=a_1^5\int A\Big(\bar{\rho}_1\Big)T\Big(n_{b, \bar{\rho}_1}(r_b)\Big) + a_2^5\int A\Big(\bar{\rho}_2\Big)T\Big(n_{b, \bar{\rho}_a}(r_b)\Big).
\end{align*}
The condition (A4) implies the similar estimate:
\begin{equation*}
\begin{aligned}
\int\frac{\rho L\Big(m_{\rho}(\eta)\Big)}{\eta^2}
&\geq\int\frac{\rho_1 L\Big(m_{\rho_1}(\eta)\Big)}{\eta^2}+\int\frac{\rho_2 L\Big(m_{\rho_2}(\eta)\Big)}{\eta^2}\\
&\geq \int\frac{\bar{\rho}_1(a_1^{-1}\cdot) a_1^4 L\Big(m_{\bar{\rho}_1}(\eta(a_1^{-1}\cdot))\Big)}{a_1^2[\eta(a_1^{-1}\cdot)]^2}
+\int\frac{\bar{\rho}_2(a_2^{-1}\cdot) a_2^4 L\Big(m_{\bar{\rho}_2}(\eta(a_2^{-1}\cdot))\Big)}{a_2^2[\eta(a_2^{-1}\cdot)]^2}\\
&\geq a_1^5\int\frac{\bar{\rho}_1 L\Big(m_{\bar{\rho}_1}(\eta)\Big)}{\eta^2}
+a_2^5\int\frac{\bar{\rho}_2 L\Big(m_{\bar{\rho}_2}(\eta)\Big)}{\eta^2}.
\end{aligned}
\end{equation*}
\begin{equation*}
\begin{aligned}
\int \rho B\rho = a_1^5 \int \bar{\rho}_1 B\bar{\rho}_1 +a_2^5\int \bar{\rho}_1 B\bar{\rho}_1 +2\int \rho_1 B\rho_2.
\end{aligned}
\end{equation*}
Therefore, set $F_{b}=\min\limits_{\rho\in W^b}E_{b}(\rho)$ and then $F_b<0$, 
\begin{equation*}
\begin{aligned}
E_b(\rho)&\geq a_1^5 E_b(\bar{\rho}_1)+ a_2^5 E_b(\bar{\rho}_2)-\int \rho_1 B\rho_2\\
&\geq  F_b-(1-a_1^5-a_2^5)F_b-\int \rho_1 B\rho_2\\
&\geq F_b-\frac{1}{C}M_1M_2 F_b-\int \rho_1 B\rho_2,
\end{aligned}
\end{equation*}
where one has used the inequality $x (1-x)\leq C [1-x^{5/3}-(1-x)^{5/3}]$ for $0\leq x\leq 1$.
\begin{equation*}
\begin{aligned}
\int \rho_1 B\rho_2&=\int(\rho \chi_{|x|\leq \frac{r}{2}}+\rho \chi_{ \frac{r}{2}<|x|\leq r}) B(\rho\chi_{|x|> r})\\
&=\iint\frac{\Big(\rho \chi_{|x|\leq \frac{r}{2}}\Big)(x)\Big(\rho\chi_{|y|> r}\Big)(y)}{|x-y|} \ dx dy+\frac{1}{4\pi}\int\nabla B(\rho \chi_{ \frac{r}{2}<|x|\leq r})\cdot \nabla B(\rho\chi_{|x|> r})\\
&\leq Cr^{-1}+C\|\nabla B(\rho \chi_{ \frac{r}{2}<|x|\leq r})\|_2 M^{1/3}\|\rho \|_{4/3}^{2/3},
\end{aligned}
\end{equation*}
where the definition of $B$ \eqref{B} and the inequality \eqref{ineq2} are used.
Then it follows that
\begin{equation*}
-F_bM_1M_2\leq Cr^{-1}+C\|\nabla B(\rho \chi_{ \frac{r}{2}<|x|\leq r})\|_2+C(E_b(\rho)-F_b)
\end{equation*}
Now applying the above inequality to the sequence $\{\rho_{R_j} \}_{j\geq1}$,
Since $r>2r_0$, one can get by  \cref{lem11}  that $M_{R_j, 1}\geq \epsilon_0$.
And then
\begin{equation*}
\begin{aligned}
-F_b\epsilon_0M_{R_j,2}&\leq Cr^{-1}+C\|\nabla B(\rho_{R_j} \chi_{ \frac{r}{2}<|x|\leq r})\|_2+C(E_b(\rho_{R_j})-F_b)\\
&\leq Cr^{-1}+C\|\nabla B(\rho_{R_j} \chi_{ \frac{r}{2}<|x|\leq r})-\nabla B(\rho_0 \chi_{ \frac{r}{2}<|x|\leq r})\|_2\\
&\qquad +C\|\nabla B(\rho_0 \chi_{ \frac{r}{2}<|x|\leq r})\|_2+C(E_b(\rho_{R_j})-F_b)
\end{aligned}
\end{equation*}
\cref{prop4} and \cref{lem12} show that the second term and the forth term in the right hand side are small for $R_j$ large enough.
$\|\nabla B(\rho_0 \chi_{ \frac{r}{2}<|x|\leq r})\|_2$\ tends to $0$ as $r\rightarrow \infty$  if  $\|\nabla B\rho_0 \|_2$ is bounded.
Therefore, $M_{R_j,2}$ is uniformly small by choosing $r$ large enough and this proves the lemma.
\end{proof}

\begin{lem}
\label{lem14}
If the condition (A7) holds, then there exists a $l<0$ such that $\lambda_{R_j}\leq l<0$ for all $R_j>2r_1$ for some $r_1>2r_0>0$.
\end{lem}
\begin{proof}
By  \cref{lem13}, choosing $r_1>2r_0$ large such that 
\begin{equation*}
\int_{|x|\geq r_1}\rho_{R_j}(x)\ dx\leq \delta_0, \qquad \text{for} \ \text{all}\ R_j.
\end{equation*} 
Hence,  (A7) gives that
\begin{equation*}
T'\Big(n_{b,\rho_{R_j}}(r_b(x))\Big) \leq 0 \qquad \text{for} \ \text{all}\ |x|\geq r_1.
\end{equation*}
Suppose that $R_j\geq \max\{\pi_1,2r_1\}$. 
Set $x_{R_j}$ as in \cref{lem10}, and then $|x_{R_j}|+1\leq r_0<\frac{r_1}{2}$.
Note that $A'(\rho)=o(\rho^{1/3})$ as $\rho\rightarrow 0$.
Now for any $x$ satisfying $|x|<2r_1$, $\eta(x)>r_1$ and $\rho_{R_j}(x) \leq C r_1^{-3}$, $C$ independent of $r_1$,
\begin{equation*}
\begin{aligned}
&\qquad A'(\rho_{R_j})=o(\rho_{R_j}^{1/3})=o(r_1^{-1})\\
&\int_{r_b(y)>r_b(x)}A(\rho_{R_j}(y))T'\Big(n_{b,\rho_{R_j}}(r_b(y))\Big)\ dy \leq 0\\
&\int_{\eta(x)}^{\infty}L(m_{\rho_{R_j}}(s))s^{-3}\ ds\leq \frac{1}{2}\sup L(m) \cdot \frac{1}{4}r_1^{-2}=O(r_1^{-2})\\
&B\rho_{R_j}(x)\geq \int_{|y-x_{R_j}|<1}\frac{\rho_{R_j}(y)}{|y-x|}\ dy\geq \dfrac{\epsilon_0}{r_0+2r_1}=O(r_1^{-1}).
\end{aligned}
\end{equation*}
So it follows  the inequality \eqref{eprime1r} that 
\begin{equation*}
\lambda_{R_j} \leq o(r_1^{-1})+O(r_1^{-2})-O(r_1^{-1})
\end{equation*}
Fixing $r_1$ yields  \cref{lem14}.
\end{proof}

\begin{proof}[\textbf{Proof of  \cref{thm2}}]
It follows from \eqref{eprime2r} and  \cref{lem14} that
\begin{equation*}
\label{eprime3r}
\begin{aligned}
 A'(\rho_{R_j}(x))T(n_{b,\rho_{R_j}}(r_b(x)))&\leq 
-\int_{r_b(y)>r_b(x)}A(\rho_{R_j}(y))T'(n_{b,\rho_{R_j}}(r_b(y)))\ dy\\
&\qquad\qquad+B\rho_{R_j}(x)+l.
\end{aligned}
\end{equation*}
To prove that there is a $R_0$ such that $\rho_{R_j}=0$ outside the ball $B_{R_0}$, it suffices to prove that $-\int_{r_b(y)>r_b(x)}A(\rho_{R_j}(y))T'(n_{b,\rho_{R_j}}(r_b(y)))\ dy\leq -l/2$ and $B\rho_{R_j}(x)\leq-l/2$ when $|x|>R_0$.

We first prove that $-\int_{r_b(y)>r_b(x)}A(\rho_R(y))T'(n_{b,\rho_R}(r_b(y)))\ dy\leq -l/2$ for $|x|$ large enough.
By (A1) and (A2), there exists a constant $s_c'$ such that 
\begin{equation*}
A(s)\leq\begin{cases}
 Cs^{4/3} & s\leq s_c'\\
 Cs^{\bar{\gamma}} & s>s_c'
\end{cases},
\end{equation*}
where $C$ is a new constant. 
Thus, 
\begin{equation*}
\begin{aligned}
\int_{r_b(y)>r_b(x)} A(\rho_{R_j}(y))& T'(n_{b,\rho_{R_j}}(r_b(y)))\  dy \leq T_0\int_{r_b(y)>r_b(x)} A(\rho_{R_j})\\
 &\leq
CT_0(\|\rho_{R_j}\|_{\infty}^{1/3}+\|\rho_{R_j}\|_{\infty}^{\bar{\gamma}-1})\int_{r_b(y)>r_b(x)}\rho_{R_j},
\end{aligned} 
\end{equation*}
where $C$ is a constant independent of $\rho_{R_j}$.
Then  \cref{lem13} implies that there is a $r_2>2r_0>0$ such that for all $|x|\geq r_2$  
$$
-\int_{r_b(y)>r_b(x)}A(\rho_{R_j}(y))T'(n_{\rho_{R_j}}(r(y))) dy\leq -l/2 \qquad \text{for all} \ j. $$ 

$B\rho_{R_j}(x) \leq -l/2$ follows the proof of Auchmuty and Beals \cite{Auchmuty1971}, and for completeness we present it here.
Let 
$$
\epsilon_{r_3, R}=\sup\limits_{|x|>r_3/2} \int_{|y-x|<1}\rho_R(y)\ dy.
$$
A similar argument to  \cref{lem11} gives that $\epsilon_{r_3,R}=O(r_3^{-1})$ as $r_1\rightarrow\infty$. Then for $|x|>r_3$ and $1<r<\frac{1}{2}r_3$, we have
\begin{align*}
B\rho_R(x)&=\int\limits_{|x-y|<1}+\int\limits_{1\leq|x-y|<r}+\int\limits_{|x-y|\geq r}\\
&\leq C(r_3^{-\alpha_1}+r_3^{-\alpha_2})+Cr_3^{-1}r^3+Mr^{-1}.
\end{align*}
Note that the shell $1<|y-x|<r$ can be covered by no more than $Cr$ balls of radius $1$. 
Choosing $r$ large and then $r_3$ large, we get that $B\rho_{R_j}(x)\leq -l/2$ when $|x|\geq r_3$ for all $R_j$.

As a result, there exists a $R_0=\max\{2r_1, r_2, r_3\}$ such that $\{\rho_{R_j}\}_{j\geq1}\subset W_{R_0}^b$. 
The proof of  \cref{lem4} shows that $E_b$ is weakly lower semicontinuous on $W_{R_0}^b$. 
Hence the weakly limit $\rho_0$ is also in $W_{R_0}^b$ and $\lim\limits_{j\rightarrow\infty}E_b(\rho_{R_j})=E_b(\rho_0)$. 
So  \cref{lem13} yields that $\rho_0$ is the compactly supported minimizer in $W^b$. (Actually, any $\rho_{R_j}$ satisfy $R_j\geq R_0$ is the minimizer in $W_{R_0}^b$.) 
Consequently, Theorem 2 is proved.
\end{proof}
\clearpage


\section{An Improved Existence Theorem}
 
In this section, we will improve Theorem 2 by enlarging the admissible set $W^b$ to $\bigcup\limits_{ b\in[\frac{1}{\xi},\xi] } W^b$.
There is also an conjecture to slightly improve \cref{thm3} in \cref{rem-conjecture} in the end of section 5.

\cref{thm2} shows that for any $b>0$ there exists a $\rho_b\in W^b$ which minimizes $E_b$, and there exists a constant $R_b$ such that such $\rho_b$ vanishes for $|x| \geq R_b$. 
The minimum of $E_b$ on $W^b$ can be considered as a function of $b$. 
We denote it as $F_b=E_b(\rho_b)$ and thus $F_b$ represents the energy minimum for $E_b$.
To prove Theorem 3, it suffices to prove that 
\begin{equation}
\label{3.3}
F_b \ \text{is continuous with respect to} \ b \ \text{in} \ [\frac{1}{\xi},\xi].
\end{equation}

\begin{cor}
\label{lem16} Fix $b_0\in [\frac{1}{\xi},\xi]$ and let $ \Gamma=\{b\in[\frac{1}{\xi},\xi]~|~ F_b\leq F_{b_0}\}$.
Then there exists a $K_0<\infty$ such that 
\begin{equation*}
\max\limits_{b\in \Gamma} \{\| \rho_b\|_{4/3} \}\leq K_0
\end{equation*} 
\end{cor}
\begin{proof}
Note that for any $b\in\Gamma$, $F_b=E_b(\rho_b)\leq F_{b_0}<0$. So the proof of  \cref{lem5} gives the desired conclusion which completes the proof of the corollary.
\end{proof}

\begin{cor}
\label{lem17}
Define 
$$
\rho^*:=\max\limits_{b\in \Gamma} \{\| \rho_b\|_{\infty} \}.
$$
Then $\rho^*<\infty.$
\end{cor}
\begin{proof}
The proof follows from   \cref{lem6} and \cref{lem7} by replacing the term $E_b(\rho_{\pi_0})$ and $k_0$ with $F_{b_0}$ and $K_0$ respectively.
\end{proof}

\begin{cor}
\label{lem18}
For any $b\in \Gamma$,
$$\|B\rho_b \|_{\infty}\geq -2F_{b_0}/M,$$ 	
and there exists a uniform constant $R_0$ such that any $\rho_b$ obtained in  \cref{thm1} vanishes for $|x|\geq R_0$.
\end{cor}
\begin{proof}
A similar argument to  \cref{lem9} gives the lower bounds of $\|B\rho_b\|_{\infty}$.
With this and $b\in\Gamma\subset[\frac{1}{\xi},\xi]$, the proof of Theorem 2 shows that $R_0$ can be uniformly bounded.
\end{proof}

\begin{proof}[\textbf{Proof of  \cref{thm3}}] 
Now  we will prove the continuity of $F_b$.
Fix $b\in\Gamma\subset[\frac{1}{\xi},\xi]$. 
Assume that $\rho\in W^b$ and $\bar{b}\in[\frac{1}{\xi},\xi]$, 
and define $a=b/\bar{b}$,
$\bar{\rho}=a\rho(x')$, where $x=(x_1,x_2,x_3)$ and $x'=(x_1,x_2,ax_3)$.
Thus, $\rho(x)=\rho(r_b(x))$, $\bar{\rho}(x)=\bar{\rho}(r_{\bar{b}}(x))$ and $\bar{\rho}\in W^{\bar{b}}$. 
Denote 
\begin{align*}
E_{b,1}&:=\int A(\rho(x))T\Big(n_{b,\rho}(r_b(x))\Big) dx;\\
E_{\bar{b},1}&:=\int A(\bar{\rho}(x))T\Big(n_{\bar{b},\bar{\rho}}(r_{\bar{b}}(x))\Big)dx
=\int A(a\rho(x'))T\Big(n_{b,\rho}(r_b(x'))\Big) a^{-1}dx';\\
E_{b,2}&:=\int \rho(x)L\Big(m_{\rho}(\eta(x))\Big)\eta^{-2}(x)dx;\\
E_{\bar{b},2}&:=\int \bar{\rho}(x)L\Big(m_{\bar{\rho}}(\eta(x))\Big)\eta^{-2}(x) dx\\
&=\int a\rho(x')L\Big(m_{\rho}(\eta(x'))\Big)\eta^{-2}(x') a^{-1} dx'=E_{b,2};
\end{align*}
\begin{align*}
E_{b,3}&:=\iint \dfrac{\rho(x)\rho(y)}{|x-y|} dx dy
=\int\dfrac{\rho(x')\rho(y')}{|x'-y'|}dx'dy';\\
E_{\bar{b},3}&:=\iint \dfrac{\bar{\rho}(x)\bar{\rho}(y)}{|x-y|} dx dy=\iint \dfrac{\rho(x')\rho(y')}{|x-y|}dx'dy'.
\end{align*}
Note that $A'(s)\leq A'(\rho^*)$ when  $s\in[0,\rho^*]$ and $|\frac{1}{|x-y|}-\frac{1}{|x'-y'|}|\leq C_1\frac{1}{|x'-y'|}$, where $C_1$ is a continuous function of $a$. 

Now set $\rho=\rho_b$. It holds that
\begin{equation}
\label{3.4}
\begin{aligned}
|E_{\bar{b},3}-E_{b,3}|& \leq C_1 E_{b,3}\\
&\leq C_1'M^{2/3}K_0^{4/3}\\
|E_{\bar{b},1}-E_{b,1}|& \leq T_0\int C_2 |A(a\rho(x))-A(\rho(x))| dx\\
& \leq C_2A'(\rho^*) T_0 \int |a-1|\rho^* dx\\
&\leq C_2'A'(\rho^*)T_0 \rho^* (\frac{4}{3}\pi R_0^3),
\end{aligned}
\end{equation}
here $C_1, C_1', C_2$ and $C_2'$ are all continuous functions of $a$. 
This implies that for any $\epsilon>0$, there exists $\delta>0$ such that any $\bar{b}$ satisfying $|b-\bar{b}|<\delta$, 
$$
F_{\bar{b}}-F_{b}= F_{\bar{b}}-E_b(\rho_b)\leq E_{\bar{b}}(\bar{\rho})-E_{b}(\rho)<\epsilon.
$$
Note that the terms in the right hand side of equation \eqref{3.4} do not depend on $\rho_b$, but depend only on $a$.
So let $\bar{\rho}=\rho_{\bar{b}}$. 
Then a similar argument shows that
$$F_{\bar{b}}-F_{b}\geq E_{\bar{b}}(\bar{\rho})-E_{b}(\rho)>-\epsilon.$$
Therefore the continuity of $F_b$ on $\Gamma$ is proved and thus  \cref{thm3} is proved.
\end{proof}

\begin{remark}\label{rem-conjecture}
If the following claim holds,
then the condition $b\in [\frac{1}{\xi},\xi]$ in \eqref{W} can be extended to $(0,\infty)$.	

\begin{enumerate}
	\item[] 
	\begin{tabular}[t]{|p{13cm}}
		\textbf{Claim}: The maximum $\|\int_{R^3}\frac{\rho(y)}{|x-y|} dy\|_{\infty} $
		subject to the constraints of $\rho$
		\begin{equation*}
			\ 0\leq \rho(y)\leq \rho^*, \ \rho(x)=\rho(r_b(x)),\ \int \rho(y) dy =M.
		\end{equation*}
		is attained by $x=0$, $\rho=\rho^* \chi_{B_0}$
		, where $B_0$ is an ellipsoid with center $(0,0,0)$ and semi-principal axes of length $r_b^*$, $r_b^*$ and $br_b^*$, and $\frac{4}{3}\pi b(r_b^*)^3\rho^*=M$.	
	\end{tabular}
\end{enumerate}
		
Indeed a direct calculation using changing variables gives the maximum $I^*$:
        \begin{equation*}
		I^*=\int_{R^3}\frac{\rho^*\chi_{B_0}(y)}{|y|} dy=\begin{cases}
		\left(\dfrac{9\pi M^2\rho^*}{2} \right)^{1/3}\dfrac{b^{1/3}}{\sqrt{b^2-1}}\ln \Big(\sqrt{b^2-1}+b\Big) & \text{if}\ b>1\\
		\left(\dfrac{9\pi M^2\rho^*}{2} \right)^{1/3}& \text{if}\ b=1\\
		\left(\dfrac{9\pi M^2\rho^*}{2} \right)^{1/3}\dfrac{b^{1/3}}{\sqrt{1-b^2}}\arcsin\left(\sqrt{1-b^2} \right) & \text{if}\ b<1
		\end{cases}.
		\end{equation*}
$I^*$ tends to $0$ as $b$ tends to $0$ or $\infty$, which contradicts to  \cref{lem18}. 
Therefore, the bound $[\frac{1}{\xi},\xi]$ for $b$ is naturally satisfied if the above claim holds.  
\end{remark}

%
\textbf{Acknowledgement} This work was supported by National Natural Science Foundation of China (Grant No. 11901208, 11971009).
The author is grateful to Professor Zhouping Xin and Professor Tao Luo for their constructive suggestions on this work.


\bibliographystyle{amsplain}
\bibliography{rotating_star.bib}




\end{document}